\documentclass[reqno, 11 pt]{amsart}

\usepackage[a4paper, centering]{geometry}
\usepackage{amsmath,amssymb, amsthm,amsfonts}
\usepackage{mathabx} 
\usepackage[T1]{fontenc}
\usepackage{verbatim} 
\usepackage{enumerate}

\theoremstyle{plain}
\newtheorem{thm}{Theorem}
\newtheorem{prop}[thm]{Proposition}
\newtheorem{lem}[thm]{Lemma}
\newtheorem{cor}[thm]{Corollary}
\newtheorem{ass}[thm]{Assumption}
\newtheorem{exm}[thm]{Example}
\newtheorem{claim}{Claim}

\newtheorem*{notn}{Notation}

\theoremstyle{definition}
\newtheorem{rmk}[thm]{Remark}
\newtheorem*{rmk*}{Remark}

\newcommand{\tagR}{\stepcounter{equation}\tag{R\theequation}}
\newcommand{\eps}{\varepsilon}
\newcommand{\wbar}{\widebar}
\newcommand{\MM}[4]{ \begin{pmatrix} #1 & #2 \\ #3 & #4 \end{pmatrix}}
\newcommand{\M}[4]{ \left( \begin{smallmatrix} #1&#2\\ #3&#4 \end{smallmatrix} \right)}
\newcommand{\kk}{k}
\newcommand{\GF}{\mathbb{F}}
\newcommand{\mf}{\mathfrak}
\newcommand{\cat}{\hat{\mathcal{C}}}
\newcommand{\artcat}{\mathcal{C}}
\newcommand{\rhobar}{\wbar{\rho}}
\newcommand{\maxId}[1]{\mf{m}_{#1}}
\newcommand{\ifThen}[2]{\text{If }&&  &#1  &&\text{ then } &&#2 \tagR}

\DeclareMathOperator{\im}{im}
\DeclareMathOperator{\Hom}{Hom}
\DeclareMathOperator{\CHom}{CHom}
\DeclareMathOperator{\End}{End}
\DeclareMathOperator{\Ad}{Ad}
\DeclareMathOperator{\Char}{char}
\DeclareMathOperator{\GL}{GL}
\DeclareMathOperator{\SL}{SL}
\DeclareMathOperator{\PSL}{PSL}
\DeclareMathOperator{\Witt}{W}
\DeclareMathOperator{\Ob}{Ob}
\DeclareMathOperator{\Def}{Def}
\DeclareMathOperator{\Sets}{Sets}
\DeclareMathOperator{\tr}{tr}

\author{Krzysztof Dorobisz}
\title[The inverse problem for univ. def. rings and the special linear group]{The inverse problem for universal deformation rings and the special linear group}

\address{Mathematisch Instituut, Universiteit Leiden, P.O. Box 9512, 2300~RA Leiden, The Netherlands.}
\email{dorobiszkj@math.leidenuniv.nl}

\begin{document}

\begin{abstract}
We present a solution to the inverse problem for universal deformation rings of group representations. Namely, we show that every complete noetherian local commutative ring $R$ with a finite residue field $\kk$ can be realized as the universal deformation ring of a continuous linear representation of a profinite group. More specifically, $R$ is the universal deformation ring of the natural representation of $\SL_n(R)$ in $\SL_n(\kk)$, provided that $n\geq 4$. We also check for which $R$ an analogous result is true in case $n=2$ and $n=3$. 
\end{abstract}

\maketitle

\section{Introduction}

The inverse problem for universal deformation rings of group representations is the following question: for which complete noetherian local commutative rings $R$ with finite residue field $\kk$ is there a profinite group $G$ and a continuous group representation $\wbar{\rho}: G \rightarrow \GL_n(\kk)$ such that $R$ is the universal deformation ring of the resulting deformation functor? 
So far several questions on the structure of the universal deformation rings, including the one of Matthias Flach, who wondered whether the deformation rings need to be complete intersections, have been answered negatively (\cite{BleherChinburg1,BleherChinburg2,Byszewski,BleherChinburgDeSmit,Rainone}). Moreover, no example of a ring $R$ that satisfies the assumptions of the problem statement and is not a universal deformation ring has been found. In this paper we prove that in fact every $R$ can be realized this way. 

The example we use for this goal is the special linear group $G := \SL_n(R)$ together with the natural representation (induced by the reduction $R\twoheadrightarrow \kk$) in $\GL_n(\kk)$, with the assumption $n \geq 4$. This is the main result of the paper. We moreover discuss similar representations for $n=2,3$ and the results of our considerations may be summarized as follows:

\begin{thm}\label{IntrThm}
Let $R$ be a complete noetherian local ring with a finite residue field $\kk$, $n\geq 2$ and consider the natural representation $\wbar{\rho}$ of $\SL_n(R)$ in $\GL_n(\kk)$. Then $R$ is the universal deformation ring of $\wbar{\rho}$ if and only if $(n, \kk) \not \in \{  (2, \GF_2) , (2, \GF_3) , (2, \GF_5),  (3,\GF_2)  \}.$
\end{thm}

We conclude the paper by discussing deformations of the natural representations of the closed subgroups of $\GL_n(R)$ containing $\SL_n(R)$. In particular, in the case of the general linear group we obtain $R$ as the universal deformation ring only when $R=\kk$ and $(n, \kk) \not \in \{  (2, \GF_2) , (2, \GF_3) , (3,\GF_2)  \}$; in some cases a noetherian universal deformation ring even does not exist. On the other hand, considering the group $\{ A \in \GL_n(R) \ | \ (\det A)^{\#\kk-1}=1\}$ we realize $R$ as the universal deformation ring if and only if  $(n, \kk) \not \in \{  (2, \GF_2) , (2, \GF_3),  (3,\GF_2)  \}.$		

\begin{rmk*}
The author has announced the main result of the paper during the Maurice Auslander Distinguished Lectures and International Conference in Woods Hole in April 2013. Similar results have been obtained independently by T. Eardley and J. Manoharmayum in their preprint \cite{EardleyManoharmayum}, published at almost the same time as the first ArXiv preprint version of this manuscript (August 2013). However, the methods of both papers are different. The reader is encouraged to get familiar also with the approach of Eardley and Manoharmayum, which is based on cohomology computations. We present a more elementary and self-contained approach treating also some cases ($n=2$; $n=4$, $\kk = \GF_2$) that \cite{EardleyManoharmayum} does not cover. 
\end{rmk*}

\section{Notation and conventions}

In this paper $\kk$ will denote a finite field, $p$ its characteristic and $\Witt(\kk)$ the ring of Witt vectors over~$\kk$. All commutative rings are assumed to have a unit.

\subsection{Categories $\cat$ and $\artcat$} Fixing $\kk$ we consider the category $\cat$ of all complete noetherian local commutative rings with residue field $\kk$. Morphisms of $\cat$ are the local ring homomorphisms inducing the identity on $\kk$. Note that all objects of $\cat$ have a natural structure of a $\Witt(\kk)$-algebra and morphisms of $\cat$ coincide with local $\Witt(\kk)$-algebra homomorphisms. 

If $R \in \Ob(\cat)$ then $\mf{m}_R$ denotes its maximal ideal and $\mu_R = \{ x\in R \ | \ x^{\# \kk-1}=1\}$ the image of the Teichm\"uller lift of $\kk^\times$. By $R_{1}^\times$ we denote the multiplicative subgroup $1+\mf{m}_R$ of $R^\times$ and by $h_R : \cat \rightarrow \Sets$ the functor $\Hom_{\cat}(R,-)$.

By $\artcat$ we denote the full subcategory of artinian rings in $\cat$. Note that every $R \in \Ob(\cat)$ is an inverse limit of $R/\mf{m}_R^l \in \Ob(\artcat)$, $l \in \mathbb{N}$.  

\subsection{Matrices} For $n \in \mathbb{N}$ and an abelian group $A$ the additive group of $n\times n$ matrices over $A$ will be denoted by $M_n(A)$. We adopt the convention of writing $M(i,j)$ for the $(i,j)$-entry of $M \in M_n(A)$. We reserve the symbol $I$ for the identity matrix.

\subsection{Profinite groups} Given a topological ring $R$ and a profinite group $G$ we denote by $R[[G]]$ the inverse limit of group rings $R[G/N]$, $N$ ranging over all open normal subgroups of $G$. Furthermore, $G^p$ and $G^{ab,p}$ stand for the pro-$p$ (abelianized pro-$p$) completion of $G$, i.e.,~$\varprojlim_N G/N$, the limit taken over all open normal subgroups $N$ with $G/N$ a $p$-group (an abelian $p$-group). If $G$ and $H$ are topological groups then $\CHom(G,H) := \{ f \in \Hom(G,H) \ | \ f \hbox{ continuous} \}$. 

\subsection{Group representations}
Given a group representation $\rho: G \rightarrow \GL_n(R)$ we denote by $V_\rho$ the corresponding module over the group ring $RG$, free of rank $n$ over $R$. That is, $V_\rho$ is the module of $n\times 1$ column vectors over $R$, on which $g \in G$ acts via multiplication by $\rho(g)$. 

By $\Ad(\rho)$ we mean the free $R$-module $M_n(R)$ on which $g \in G$ acts via conjugation with $\rho(g)$. Thus $\Ad(\rho) \cong \End_{\kk}(V_{\rho})$. We will denote by $\Ad(\rho)^G$ the submodule of $G$-invariants.

\section{Deformations of group representations}

This quick review of the deformation theory of group representations is based on Mazur's paper \cite{Mazur}. We also recommend a very good and detailed exposition of the topic by Gouvea, \cite{Gouvea}. A reader familiar with the deformation theory should only check our definition of the natural transformation $\rho_{*}$ in section~\ref{IntroDefo}, not used elsewhere in the literature, as well as the lemmas in section~\ref{IntroTechLemmas}.

For the whole of this section $G$ will be a profinite group and $\rhobar : G \rightarrow \GL_n(\kk)$ a continuous representation.

\subsection{Deformations and deformation rings} \label{IntroDefo}

Let $R \in \Ob(\cat)$ and $\pi_R : \GL_n(R) \rightarrow \GL_n(\kk)$ be induced by the reduction $R \rightarrow \kk$. By a \textbf{lift} of $\wbar{\rho}$ to $R$ we mean any $\rho \in \CHom(G, \GL_n(R))$ such that $\wbar{\rho} = \pi_R \circ \rho$. Two lifts to $R$, $\rho$ and $\rho'$, are called ``\textbf{strictly equivalent}'' if and only if $\rho' = K \rho K^{-1}$ for some $K \in \ker \pi_R$. We denote the strict equivalence class of $\rho$ by $[\rho]$ and the resulting set of all equivalence classes, called \textbf{deformations}, by $\Def_{\wbar{\rho}}(R)$. 

Note that if $\rho$ is a lift of $\rhobar$ to $R \in \Ob(\cat)$ then $V_{\rho}$ is such an $RG$-module that $\kk \otimes_R V_{\rho} \cong V_{\rhobar}$ as $\kk G$-modules. The $RG$-modules corresponding to strictly equivalent lifts are isomorphic.

Every $\cat$-morphism $f:R\to R'$ induces a group homomorphism $\GL_n(f) : \GL_n(R) \to \GL_n(R')$  (defined by applying $f$ to each of the entries of a given matrix) and a map $\Def_{\wbar{\rho}}(f) : \Def_{\wbar{\rho}}(R) \rightarrow \Def_{\wbar{\rho}}(R')$, $[\rho] \mapsto [\GL_n(f)\circ\rho]$. This way the \textbf{deformation functor} $\Def_{\wbar{\rho}}: \hat{\mathcal{C}} \to \Sets$ is obtained. We are interested in the question of its representability, i.e.,~in the existence of $R \in \Ob(\cat)$, such that $\Def_{\rhobar}$ and $h_R$ are naturally isomorphic. If this is the case, $R$ is called the \textbf{universal deformation ring} of $\wbar{\rho}$. Note that such a ring is unique up to an isomorphism by the Yoneda lemma.

Every lift $\rho : G \rightarrow \GL_n(R)$ of $\rhobar$ induces a \textbf{natural transformation} $\rho_{*} : h_R \rightarrow \Def_{\rhobar}$ defined as $f \mapsto [\GL_n(f) \circ \rho]$. If this transformation is an isomorphism we say that $\rho_u$ is a \textbf{universal lift} (in this case $R$ clearly is a universal deformation ring of $\rhobar$). We will write $\rho_{*}(S)$ for the map $h_R(S) \rightarrow \Def_{\rhobar}(S)$ given by $\rho_{*}$. 

In some situations a notion weaker than representability is useful. We say that a functor $D : \cat \rightarrow \Sets$ is nearly representable if there exists $R \in \Ob(\cat)$ and a natural transformation $h_{R} \rightarrow D$ that is \emph{smooth} and bijective for $\kk[\eps] := \kk[X]/(X^2)$. For the meaning of term ``smooth'' we refer the reader to Mazur (\cite[\S 18]{Mazur}); we shall only use the fact that it implies that $h_{R}(S) \rightarrow D(S)$ is surjective for every $S \in \Ob(\cat)$. If $D = \Def_{\rhobar}$ is nearly representable then a ring $R$ as above will be called the versal deformation ring of $\rhobar$. If this is the case then such a ring is unique up to an isomorphism and if the functor is representable then the versal and universal deformation rings are isomorphic. 

\subsection{Existence theorems}\label{IntroExistCrit} Following Mazur we introduce the so-called ``$p$-finiteness condition'':
\[(\Phi_p) \quad \hbox{ For every open subgroup } J \leq G \hbox{ the set }\CHom (J, \ \mathbb{Z}/p\mathbb{Z}) \hbox{ is finite}.\]
See \cite[Lemma 2.1]{Gouvea} for equivalent formulations. Assuming that $G$ satisfies $\Phi_p$ the following results hold (\cite[\S 20]{Mazur}, \cite{deSmitLenstra}):
\begin{enumerate}
\item The versal deformation ring of $\rhobar$ exists.
\item If $\Ad(\rhobar)^G = \kk I_n$ then the universal deformation ring of $\rhobar$ exists. 
\end{enumerate}
We note that the second result holds (with some modifications) also without assuming $\Phi_p$. Namely, in general we obtain representability in a larger category, by a local $\Witt(\kk)$-algebra which, however, might not be noetherian. Assuming condition $\Phi_p$ is sufficient to guarantee that this algebra is noetherian and complete. See~\cite{deSmitLenstra} or \cite[\S 2.3]{Hida} for details.

\subsection{Tangent space} 

For $R\in \Ob(\cat)$ we define the tangent space $t_R$ to be $h_R(\kk[\eps])$, where $\kk[\eps] = \kk[X]/(X^2)$. Similarily, for a deformation functor $D = \Def_{\wbar{\rho}}$ the tangent space $t_D$ to $D$ is defined as $D(\kk[\eps])$. Both types of tangent spaces have natural structures of $\kk$-vector spaces (\cite[\S 15]{Mazur}). Moreover, the tangent space to $D$ can be naturally interpreted as the continuous cochain cohomology group $H^1(G, \Ad(\rhobar))$, see \cite[\S 21]{Mazur}.

\begin{prop}  \ \label{PropPresentation}
\begin{enumerate}[(i)]
\item Every $R \in \Ob(\cat)$ is a quotient of $\Witt(\kk)[[X_1, \ldots, X_d]]$ for some $d$. 
\item The minimal number $d$ with the property of $(i)$ is equal to $\dim_\kk t_R$.
\item Let $D$ be a deformation functor with a versal deformation ring $R_v$. Then $R_v$ is a quotient of $\Witt(\kk)[[X_1, \ldots X_d]]$, where $d := dim_\kk (t_D)$.
\end{enumerate}
\end{prop}

\begin{proof}[Sketch of the proof]
Part (1) follows from Cohen's structure theorem, see \cite[Theorem 29.4,(ii)]{Matsumura}. The minimal number $d$ as in $(i)$ is the minimal number of generators of the maximal ideal of $R/pR$, which equals $\dim_\kk \mf{m} / (\mf{m}^2,p)$. On the other hand, morphisms $R \rightarrow \kk[\eps]$ correspond bijectively with $\kk$-linear maps $\mf{m} / (\mf{m}^2,p) \rightarrow \kk$ (\cite[\S 15]{Mazur}). The last claim follows from the definition of a versal deformation ring, due to which there is a bijection between $t_D$ and $t_{R_v}$. 
\end{proof}

\subsection{Some technical lemmas}\label{IntroTechLemmas}

The following general results will be useful in the later part of the paper. A reader wanting to learn quickly the main results may skip Lemma~\ref{CentrLemma} and Lemma~\ref{Lemma_sl-gl}, since they will be used only in the last two sections.

\begin{lem}\label{ProjLemma}
If $G$ is finite and $\rhobar$ such that the $\kk G$-module $V_{\rhobar}$ is projective then $\Witt(\kk)$ is the universal deformation ring of $\rhobar$. In particular, this applies to every representation of a finite group $G$ of order coprime to $p$.
\end{lem}
\begin{proof}
Let $R_v$ be the versal deformation ring of $\wbar{\rho}$. Since $V_{\rhobar}$ is $\kk G$-projective, $\Ad(\rhobar)$ is $\kk G$-projective as well, hence cohomologically trivial. The tangent space $H^1(G,\Ad(\rhobar))$ to $D_{\rhobar}$ is therefore zero-dimensional and so $R_v$ is a quotient of $\Witt(\kk)$ by Proposition~\ref{PropPresentation}. Hence, there is at most one deformation of $\wbar{\rho}$ to every $S \in \Ob(\cat)$; in particular: $R_v$ is universal. On the other hand, by Prop. 42, \S 14.4 in \cite{Serre} the $\kk G$-module $V_{\rhobar}$ can be lifted to a $\Witt(\kk)G$-module that is free over $\Witt(\kk)$. This implies that $R_v = \Witt(\kk)$. In case $G$ is finite and $p \nmid \# G$, every $\kk G$-module of finite $\kk$-dimension is projective by Maschke's theorem.
\end{proof}

\begin{lem}\label{CentrLemma}
If $\Ad(\rhobar)^G = \kk I$ and $\rho$ is a lift of $\rhobar$ to $R \in \Ob(\cat)$ then $\Ad(\rho)^G = R I$.
\end{lem}
\begin{proof}[Sketch of the proof, cf. {\cite[Lemma 3.8]{Gouvea}}] 
Let $X \in \Ad(\rho)^G$. It suffices to show inductively that for every $l\geq 1$ matrix $X$ is scalar modulo $\mf{m}_R^l$. For $l=1$ this is assumed in the lemma statement. In the inductive step we have $X = \lambda I + Y$, where $\lambda \in R$ and $Y \in M_n(\mf{m}_R^{l-1})$. Let $W$ be the $\kk$-vector space $\mf{m}_R^{l-1}/\mf{m}_R^l$ and $d$ its dimension. Then $M_n(W) \cong \Ad(\rhobar)^{d}$ as $\kk G$-modules and $M_n(W)^G \cong (\Ad(\rhobar)^d)^G = (\kk I)^d$. Clearly $Y \in \Ad(\rho)^G$, so its image modulo $\mf{m}_R^l$ lies in $M_n(W)^G$. It follows that $Y$, hence also $X$, is a scalar modulo $\mf{m}_R^l$.
\end{proof}

\begin{lem}\label{Lemma_sl-gl}
Let $N \lhd G$ be a closed normal subgroup such that $\Ad(\rhobar)^N = \kk I$ and $\rhobar|_N$ has a universal deformation ring $R$. Suppose there exists a universal lift of $\rhobar|_N$ that may be extended to a lift $\varphi : G \rightarrow \GL_n(R)$ of $\rhobar$.
\begin{enumerate}[(i)]
\item For every $S \in \Ob(\cat)$ and $\xi \in \Def_{\rhobar}(S)$ there exist unique $f\in \Hom_{\cat}(R, S)$ and $\lambda \in \CHom(G, S_1^\times)$ with $N \subseteq \ker \lambda$ such that $\lambda \varphi_{*}(f) \in \xi$.
 \item If $\Gamma := (G/N)^{ab,p}$ considered with its natural $\mathbb{Z}_p$-module structure is finitely generated then $R[[ \Gamma ]]$ is a universal deformation ring for $\rhobar$. Otherwise $\Def_{\rhobar}$ is not representable over $\cat$. 
\end{enumerate}
\end{lem}

\begin{proof}
$(i)$ Let $S \in \Ob(\cat)$ and $\xi \in \Def_{\rhobar}(S)$. Restricting to $N$ and using the definition of a universal lift we obtain $f \in \Hom_{\cat}(R, S)$ and $\rho \in \xi$ such that $\rho|_{N} = \varphi_{*}(f) |_{N}$. For brevity we will write $\rho_f$ for $\varphi_{*}(f)$. 

Let $g \in G$ and $n \in N$. Since $N$ is normal we have $\rho_f(gng^{-1}) = \rho(gng^{-1})=\rho(g)\rho_f (n)\rho(g)^{-1}$. Consequently, $\rho_f (g)^{-1}\rho(g)$ commutes with all $\rho_f(n)$, $n \in N$ and is therefore a scalar matrix by Lemma~\ref{CentrLemma}. Let $\lambda :  G \rightarrow S_1^\times$ be such that $\rho(g) = \lambda(g) \rho_f(g)$. Then $\lambda$ is a continuous group homomorphism factoring via $G/N$. Conversely, given $f \in \Hom_{\cat}(R, S)$ and $\lambda \in \CHom(G, S_1^\times)$ factoring via $G/N$, the map $\varphi_{f,\lambda} : g \mapsto \lambda(g) \rho_f(g)$ is a lift of $\rhobar$. 
Moreover, if $\varphi_{f, \lambda} = X\varphi_{f', \lambda'} X^{-1}$ for $X \in \GL_n(R)$ then restricting to $N$ we see that $X$ is scalar by Lemma~\ref{CentrLemma}, which implies $\varphi_{f, \lambda} = \varphi_{f', \lambda'}$, i.e., $f=f'$, $\lambda = \lambda'$. 

$(ii)$ If $\Gamma$ is a finitely generated $\mathbb{Z}_p$-module then, since $\mathbb{Z}_p$ is a PID, it is a quotient of $\mathbb{Z}_p^r$ for some $r\in \mathbb{N}$. Hence, 
$R[[\Gamma]]$ is a quotient of $R[[ \mathbb{Z}_p^r ]] \cong R[[X_1, \ldots, X_r]]$ and an object of $\cat$ (cf. \cite[Corollary 2.24]{Hida}). Denote the image of $g \in G$ under the map $G \rightarrow \Gamma \rightarrow R[[\Gamma]]$ by $[g]$. The first part of the lemma implies that $G \ni g \mapsto \rho(g) [g] \in \GL_n(R[[\Gamma]])$ is a universal lift of $\rhobar$ (cf. \cite[\S 1.4]{Mazur2}).

On the other hand, if $\Def_{\rhobar}$ is representable over $\cat$ then the tangent space $\Def_{\rhobar}(\kk[\eps])$ is finite dimensional. This and the first part of the lemma imply that $\CHom(\Gamma, \kk[\eps]^{\times}) \cong \CHom(\Gamma/p\Gamma,\kk)$ is a $\kk$-vector space of finite dimension. Consequently, so is $\Gamma/p\Gamma$. Since $\Gamma$ is a pro-$p$ group we have $\bigcap_{l\in\mathbb{N}} p^l\Gamma = \{0\}$ and a topological Nakayama's lemma (\cite[Lemma 3.2.6]{Hida2}) implies that $\Gamma$ is finitely generated over $\mathbb{Z}_p$ .
\end{proof}

\begin{rmk}\label{RmkNoetherianIsBad} The formulation of the lemma reflects the fact that $R[[ \Gamma ]]$ is not necessarily noetherian. Working not in $\cat$, but in a bigger category comprising also non-noetherian projective limits of local artinian rings we could simply state that $\Def_{\rhobar}$ is represented by $R[[\Gamma]]$. See \cite{deSmitLenstra} or \cite[\S 2.3]{Hida} for such a more general approach to deformation theory.
\end{rmk}

\section{Structure of the special linear group}

\noindent In this section $R$ denotes a commutative ring and $n$ an integer. Moreover, we assume $n\geq 2$. 

\begin{notn} \label{notnGL}
Let $a,b \in \{1,\ldots, n\}$, $a\neq b$. We introduce the following notation for some of the elements of $\GL_n(R)$. 
\begin{itemize}
\setlength{\itemsep}{5 pt}
	\item $e_{ab}$ is the matrix having only one non-zero entry, which is 1 at the $(a,b)$-th place.
	\item $t_{ab}^r :=I+re_{ab}$, for $r \in R$. 
	\item $d(r_1, \ldots, r_n)$, where $r_i \in R^\times$, is the diagonal matrix with consecutive diagonal entries $r_1, \ldots, r_n$. 
	\item $d_{ab}^r := d(r_1, \ldots, r_n)$, where $r \in R^{\times}$ and $r_a := r$, $r_b:=r^{-1}$, $r_i :=1$ for $i\neq a,b$.
    \item $\sigma_{ab}^{r} := I - e_{aa}-e_{bb} + re_{ab} -r^{-1} e_{ba}$, for $r \in R^{\times}$.
\end{itemize}
\end{notn}

\noindent This notation suppresses $n$ and $R$, but that should not cause any problems as $n$ will usually be fixed and $R$ easily deducible from the element $r$ used. 

\begin{lem} \label{RelLemma}
The following relations hold for $a,b,c,d \in \{1,\ldots,n\}$, $a\neq b$: 
\begin{align}
   \ifThen  {r,s \in R}                        
			{t_{ab}^r \ t_{ab}^s = t_{ab}^{r+s},}    \label{rel+}
\\ \ifThen  {r,s \in R \text{ and } c\neq a,b} 
			{[t_{ab}^r ,  t_{bc}^s ] = t_{ac}^{rs},} \label{rel[]}
\\ \ifThen  {r,s \in R \text{ and } \{a,c\} \cap \{b,d\} = \emptyset} 
			{[t_{ab}^r ,   t_{cd}^s ] = 1} 	\label{rel-comm}
\medskip			
\\ \ifThen  {r \in R \text { and } D=d(\lambda_1, \ldots, \lambda_n), \lambda_i \in R^\times}  
			{Dt_{ab}^rD^{-1} = t_{ab}^{ \ \frac{\lambda_a}{\lambda _b}r }} \label{rel-conj}
\medskip
\\ \ifThen  {u \in R^\times}					   
				 {\sigma_{ab}^u = t_{ab}^u \ t_{ba}^{-\frac{1}{u}} \ t_{ab}^u,} \label{rel-abba}
\\ \ifThen  {u \in R^\times}
 		    {d^u_{ab} = \sigma_{ab}^{u}\sigma_{ab}^{-1},}  \label{rel-dsigma} 
\\ \ifThen  {u \in R^\times, \ r,s \in R \text{ and } u = 1+rs}	  
			{d_{ab}^{u} = t_{ab}^r \ t_{ba}^{s} \ t_{ab}^{\frac{-r}{u}} \ t_{ba}^{-su}.} \label{rel-abba2} 
\end{align} 
\end{lem}

\begin{proof}[Sketch of the proof]
These identities follow directly from definitions and straightforward computations in which one uses the fact that $e_{ab}e_{cd} = \delta_{bc}e_{ad}$ ($\delta_{bc}$ being the Kronecker delta symbol). 
\end{proof}

\begin{lem}\label{GenLemma} Given an ideal $\mf{a} \unlhd R$ let $U_\mf{a} := \SL_n(R) \cap (I + M_n(\mf{a}))$ and $V_\mf{a}:= \langle t_{ab}^r \in \GL_n(R) \ | \ r \in \mf{a} \rangle$. If $R$ is local then $U_{\mf{a}^2}\leq V_\mf{a} \leq U_\mf{a}$. In particular, $U_R = \SL_n(R)$ is generated by all the elements of the form $t_{ab}^r$.
\end{lem}

\begin{proof}[Sketch of the proof]
The inclusion $V_\mf{a} \subseteq U_\mf{a}$ is obvious. Conversely, let $M \in U_{\mf{a}^2}$. Observe that multiplying $M$ by $t_{ab}^r$ amounts to adding a multiple of one of its rows or columns to some other.
We claim that performing such operations on $M$ we may obtain a diagonal matrix lying in $U_{\mf{a}^2}$. If $\mf{a}$ is contained in the maximal ideal $\mf{m}$ of the ring $R$ then all the diagonal entries of $M$ are invertible and we may simply cancel all other entries proceeding row by row. In case $\mf{a}=R$ every row contains an invertible element (since $\det M \not\in \mf{m}$), so each diagonal entry either is invertible or becomes such after one of the described operations. We proceed as follows: make $M(n,n)$ invertible, cancel all other entries in the $n$-th row and column, repeat the procedure recursively on the leading $(n-1) \times (n-1)$ submatrix. Every diagonal matrix in $U_{\mf{a}^2}$ may be decomposed as a product of matrices of the form $d_{ab}^r$, $r \in \mf{a}^2$. Relations (\ref{rel-abba}) and (\ref{rel-dsigma}) from Lemma~\ref{RelLemma} in case $\mf{a}=R$ and relation (\ref{rel-abba2}) in case $\mf{a}\subseteq \mf{m}$ imply that each of them is generated by some elements of the form $t_{ab}^r$, $r \in \mf{a}$. Reversing the described process we obtain a way of generating the matrix $M$.
\end{proof}

\begin{lem}\label{LemmaCommutator}
Assume $R$ is local with residue field $k$. If $n\geq 3$ or $n=2$ and $\kk \neq \GF_2, \GF_3$ then the commutator subgroup $\SL_n(R)'$ is equal to $\SL_n(R)$.
\end{lem}

\begin{proof}[Sketch of the proof]
It is sufficient to show that generators of $\SL_n(R)$ lie in $\SL_n(R)'$. To this end use Lemma~\ref{GenLemma} and suitable relations from Lemma~\ref{RelLemma}: (\ref{rel[]}) in case $n\geq 3$ or (\ref{rel-conj}) in case $n=2$, $\kk \neq \GF_2, \GF_3$.
\end{proof}

\begin{lem} \label{End1Lem}
If $M\in M_n(R)$ commutes with all $t_{ab}^1 \in \GL_n(R)$ then $M$ is a scalar matrix.
\end{lem}
\begin{proof}
The claim follows from the observation that $t_{ab}^1M = Mt_{ab}^1$ is equivalent to $e_{ab}M = Me_{ab}$, which holds if and only if $M(a,a)=M(b,b)$ and $\forall \ x \neq a, y \neq b : \ M(x,a) = M(y,b) = 0$. 
\end{proof} 

\section{The special linear group and deformations} 

\noindent Let us fix a finite field $\kk$ and work in the resulting category $\cat$. The following assumption will be made for the whole of this section.

\begin{ass} \label{Ass}
Let $R \in \Ob(\cat)$, $G:= \SL_n(R)$, $n \geq 2$ and $\rhobar : G \rightarrow \GL_n(\kk)$ be the representation induced by the reduction $R \twoheadrightarrow \kk$. Denote by $\iota$ the inclusion $G = \SL_n(R) \hookrightarrow \GL_n(R)$. Let us moreover set $\mf{I} := \{1, \ldots, n\}$ and $\mf{J} := \{ (a,b) \in \mf{I}\times \mf{I} \ | \ a\neq b \}$.
\end{ass}

Note that $G =\varprojlim \SL_n(R/\maxId{R}^k)$ is profinite and $\wbar{\rho}$ continuous. The subgroups $V_{\mf{m}_R^k} \leq G$ defined as in Lemma~\ref{GenLemma} are topologically finitely generated, $\CHom (V_{\maxId{R}^k}, \ \mathbb{Z}/p\mathbb{Z})$ is thus finite. Lemma~\ref{GenLemma} implies that these subgroups form a basis of open neighbourhoods of $G$, so it follows that $G$ satisfies the $p$-finiteness condition $\Phi_p$. Using Lemma~\ref{End1Lem} and the existence criterion of section~\ref{IntroExistCrit} we obtain that $\Def_{\rhobar}$ is representable. 

We are interested in the following question: is $\Def_{\rhobar}$ represented by $R$ ?

\subsection{General observations}

\begin{lem} 
The ring $R$ is the universal deformation ring of $\rhobar$ if and only if $\iota$ is a universal lift. 
\end{lem}

\begin{proof}
We merely have to prove the ``only if'' part. Assume $R$ is the universal deformation ring of $\rhobar$ and consider a universal lift $\rho_u : G \rightarrow \GL_n(R)$. By definition, there is an $f \in \Hom_{\cat}(R, R)$ such that $ [\GL_n(f) \circ \rho_u] = [\iota]$, i.e., $ \GL_n(f) \circ \rho_u = X \iota X^{-1}$ for some $X \in I + M_n(\mf{m}_R)$. In particular, $\forall r \in R: X t_{12}^{r} X^{-1} = I + r Xe_{12}X^{-1} \in \im (\GL_n(f)).$  If $a \in R^\times$ is the $(1,2)$-entry of $Xe_{12}X^{-1}$ then we see that $R = Ra \subseteq \im (f)$, so $f$ is surjective. Since $R$ is noetherian, $f$ is also an automorphism. As $\iota_*$ is a composition of the transformation $h_R \rightarrow h_R$, $\varphi \mapsto \varphi \circ f$ with $(\rho_u)_*$, it is a natural isomorphism, i.e., $\iota$ is universal.
\end{proof}

Motivated by the above lemma, we turn our attention to the properties of the natural transformation $h_R \rightarrow \Def_{\rhobar}$ induced by $\iota$. 

\begin{lem} \label{AuxLem} Let $S \in \Ob(\cat)$.
\begin{enumerate}[(i)]
\item The map $\iota_{*}(S) : h_R(S) \rightarrow \Def_{\rhobar}(S)$ induced by $\iota$ is injective. 
\item If $\xi \in \Def_{\rhobar}(S)$ then $\xi \in \im \iota_{*}(S)$ if and only if there exists a lift $\rho \in \xi$ satisfying the following condition: 
\[\forall (a,b)\in \mf{J}, r \in R \ \ \exists \ c_{ab}^r \in S: \quad \rho(t_{ab}^r) = t_{ab}^{c_{ab}^r}. \tag{$\diamondsuit$}\label{spec-form}\]
\end{enumerate}
\end{lem}
\begin{proof}

$(i)$ Let $S \in \Ob(\cat)$ and consider $f,g \in h_R(S)$ with $\iota_{*}(f)=\iota_{*}(g)$, i.e., $\GL_n(f)|_G = X \GL_n(g)|_G X^{-1}$ for some $X \in I + M_n(\mf{m}_S)$. In particular, $\forall r \in R: \ t_{12}^{f(r)} = Xt_{12}^{g(r)}X^{-1}$, which is equivalent to $f(r) e_{12} = g(r) X e_{12}X^{-1}$. For $r=1$ we have $Xe_{12}X^{-1} = e_{12}$ and so we conclude that $\forall r \in R : \ f(r)e_{12} = g(r)e_{12}$, i.e., $f=g$. 
\smallskip

$(ii)$ Every lift of $\rhobar$ of the form $\GL_n(f)\circ \iota$, $f \in \Hom_{\cat}(R,S)$ obviously satisfies (\ref{spec-form}). Conversely, consider $\rho$ satisfying (\ref{spec-form}) and suppose first that $\mathbf{n \geq 3}$.  Conjugating with the diagonal matrix $d(1,c_{12}^1, \ldots, c_{1n}^1) \in I + M_n(\mf{m}_S)$ we obtain a lift $\tilde{\rho}$ strictly equivalent to $\rho$. It satisfies (\ref{spec-form}) as well and in addition $\forall j \in \mf{I}\setminus\{1\}: \ \tilde{\rho}(t_{1j}^1) = t_{1j}^1$, due to Lemma~\ref{RelLemma}, (\ref{rel-conj}). We may thus suppose without loss of generality that $c_{1j}^1=1$ for all $j \in \mf{I}\setminus\{1\}$. Lemma~\ref{RelLemma}, (\ref{rel[]}) implies then that $\forall j,k \in \mf{J}$, $j,k \neq 1: c_{jk}^1 = c_{1k}^1 / c_{1j}^1 = 1$. Furthermore, for every $j\in\mf{I}\setminus\{1\}$ there exists $k\in\mf{I}\setminus\{1,j\}$, so $c_{j1}^1 = c_{1k}^1/c_{1j}^1 = 1$ as well. We conclude that $\forall (a,b)\in \mf{J}: \ c_{ab}^1=1$.

Due to Lemma~\ref{RelLemma}, (\ref{rel+}) and (\ref{rel[]}), the following relations are satisfied for all $r,s \in R$ and pairwise distinct $a,b,c \in \mf{I}$:
\[
\left\{
\begin{array}{lcl}
   c_{ab}^{r+s} &=& c_{ab}^{r} + c_{ab}^s 
\\ c_{ac}^{rs}  &=& c_{ab}^{r} \ c_{bc}^s 
\end{array} 
\right.
\]
Substituting in the second relation firstly $r=1$, then $s=1$, we see that the value $c_{ab}^{r}$ with a fixed $r \in R$  does not depend neither on $a$, nor on $b$. Denote this common value by $\varphi(r)$. We have obtained a function $\varphi : R \rightarrow S$, which is additive by the first relation and multiplicative by the second one, satisfies $\varphi(1)=1$ and for which $\varphi(r)$ and $r$ have the same image in $\kk$, i.e., $\varphi \in h_R(S)$. Since $G$ is generated by the elements of the form $t_{ab}^r$ (Lemma~\ref{GenLemma}) we conclude that $\rho = \GL_n(\varphi) \circ \iota$ and so $[\rho] \in \im \iota_{*}(S)$. 

\smallskip

Suppose now $\mathbf{n=2}$. We may similarily assume that $c_{12}^1=1$. Define $\varphi,g : R \rightarrow S$ by $\varphi(r) := c_{12}^r$ and $g(r):= c_{21}^r$. We claim that $g = \varphi$ and $\varphi \in \Hom_{\cat} (R,S)$, which clearly implies that $[\rho] = \iota_{*}(\varphi) \in \iota_{*}(h_R(S))$. Since $\varphi$ is additive by the relation (\ref{rel+}) of Lemma~\ref{RelLemma},  $\varphi(1)=1$ and for all $r \in R$ the images of $\varphi(r)$ and $r$ in $\kk$ coincide, we only need to check that $\varphi$ is multiplicative. Furthermore, it is sufficient to check multiplicativity only on $R^\times$, because of additivity of $\varphi$ and the fact that every non-invertible $r\in R$ is a sum of two invertible elements (e.g. $r= (r-1) + 1$). Similarily, it is sufficient to check that $g(r)=\varphi(r)$, for $r\in R^\times$.

Let $r \in R^{\times}$, $a:=\varphi(r)$, $b:=g(-r^{-1})$ and $\sigma_r := \M{0}{r}{-r^{-1}}{0}$. Applying relation (\ref{rel-abba}) of Lemma~\ref{RelLemma} twice, we get $t_{12}^r \ t_{21}^{-1/r} \ t_{12}^r = \sigma_r = t_{21}^{-1/r} t_{12}^r t_{21}^{-1/r}$, hence:
\begin{align*}
 \MM{1}{a}{0}{1}\MM{1}{0}{b}{1}\MM{1}{a}{0}{1} = &\rho( \sigma_r ) = \MM{1}{0}{b}{1} \MM{1}{a}{0}{1}\MM{1}{0}{b}{1} 
\\ \MM{1+ab\ }{2a+a^2b}{b}{1+ab} =  &\rho(\sigma_r  ) = \MM{1+ab}{a}{2b+ab^2}{\ 1+ab} 
\end{align*}
It follows that $a+a^2b=0$, so $ab=-1$, since $a$ is invertible. This means $\varphi(r) g(r^{-1}) = 1$ and $\rho(\sigma_r) = \M{0}{\varphi(r)}{-\varphi(r)^{-1}}{0}$. Relation (\ref{rel-dsigma}) implies now $\rho \left(\M{r}{}{}{r^{-1}}\right)  = \rho  (\sigma_r) \rho (\sigma_1)^{-1} = \M{\varphi(r)}{}{}{\varphi(r)^{-1}}$. As $R^{\times} \ni r \mapsto\M{r}{}{}{r^{-1}}$ is a group homomorphism, we conclude that $\varphi$ is multiplicative on $R^\times$. In particular, $\varphi(r) g(r^{-1}) = 1$ implies that $g(r^{-1}) = \varphi(r^{-1})$, so $g=\varphi$ on $R^\times$. This finishes the proof.
\end{proof}

\subsection{Main result}

\begin{thm}\label{MainThm}
If $n \geq 4$ then $\iota$ is universal.
\end{thm}
   
\begin{proof} 
Let $S\in \Ob(\cat)$. By Lemma~\ref{AuxLem} we only need to show that $\iota_*(S)$ is surjective, i.e., that every lift of $\wbar{\rho}$ to $S$ is strictly equivalent to $\rho_f := \GL_n(f)|_G$ for some $f \in \Hom_{\cat} (R,S)$. Moreover, we may restrict to the case $S \in \Ob(\artcat)$, since all rings in $\cat$ are inverse limits of artinian rings. 

For $S\in \Ob(\artcat)$ let $n(S)$ be the smallest $j\in\mathbb{N}$ such that $\mf{m}_S^j=0$. We proceed by induction on $n(S)$. For $n(S)=1$, i.e., $S = \kk$, the statement is obvious. For the inductive step consider $S$ with $n(S) \geq 2$, a lift $\rho : G \rightarrow \GL_n(S)$ of $\wbar{\rho}$ and set $l:=n(S)-1$. By the inductive hypothesis, we may suppose (considering a strictly equivalent lift if necessary) that $\rho$ reduced to $S/ \mf{m}^l_S$ is induced by a morphism $g : R \rightarrow S/\mf{m}_S^l$. For every $r \in R$ choose $p_r \in S$ such that $p_r \equiv g(r) \mod \mf{m}_S^l$; for $r=1$ we choose $p_1=1$. This way \[\forall \ (a,b) \in \mf{J}, r\in R: \quad \rho( t_{ab}^r ) = t_{ab}^{p_r} + M_{ab}^r \quad , \quad \hbox{for some }  M_{ab}^r \in M_{n\times n}(\mf{m}_S^l).\] 
We will analyze the structure of the matrices $M_{ab}^r$ proving a series of claims. In the calculations we use the fact that $J:=M_{n\times n}(\mf{m}_S^l)$ is such a two-sided ideal of $M_n(S)$ that $\mf{m}_SJ = J^2 = 0$. 

\begin{claim}\label{claim-comm} If $a,b,c,d \in \mf{I}$ are such that $\{a,c\} \cap \{b,d\} = \emptyset$ then for all $r,s \in R$: 
\[ p_s( M_{ab}^re_{cd} - e_{cd}M_{ab}^r) = p_r( M_{cd}^se_{ab} - e_{ab}M_{cd}^s)\] 
\end{claim}

\begin{proof}
Since $t_{ab}^r$ and $t_{cd}^s$ commute (Lemma~\ref{RelLemma}, (\ref{rel-comm})), so do their lifts. Denoting $t:= t_{ab}^{p_r}$, $M:=M_{ab}^r$, $z := t_{cd}^{p_s}$, $N:=M_{cd}^s$ we obtain:

\begin{align*}
(t + M)\ (z + N) &= (z + N)\  (t + M)\\
tz \ + \ tN \ + \ Mz &=  zt \ + \ zM \ + \ Nt  \\
Mz- zM & = Nt - tN \\
M (I+p_se_{cd})- (I+p_se_{cd}) M & = N (I+p_re_{ab}) -  (I+p_re_{ab})N \\
p_s( M_{ab}^re_{cd} - e_{cd}M_{ab}^r) & = p_r( M_{cd}^se_{ab} - e_{ab}M_{cd}^s ) 
\end{align*} \end{proof}

\begin{claim}\label{claim-commConclusion} $\forall (a,b) \in \mf{J}, r\in R: \ \ M_{ab}^r(i,j)=0$ when $i\neq a$, $j\neq b$ and $i\neq j$. \end{claim}

\begin{proof} If we fix $r \in R$ and $(a,b) \in \mf{J}$ then given $j \in \mf{I} \setminus\{b\}$ we may choose $d \in \mf{I} \setminus \{a,b,j\}$ (here we use the assumption $|\mf{I}| \geq 4$). Such $a,b,j,d$ satisfy the assumptions of Claim~\ref{claim-comm}, so we obtain $M_{ab}^re_{jd} - e_{jd}M_{ab}^r = p_r( M_{jd}^1e_{ab} - e_{ab}M_{jd}^1 )$. If $i \in \mf{I}\setminus\{a,j\}$ then a comparison of the $(i,d)$-entries of both sides of the relation shows that $M_{ab}^r(i,j) = 0$. \end{proof} 

\begin{claim}\label{claim-tr} $\forall (a,b) \in \mf{J}, r\in R: \ \ \tr M_{ab}^r=0$. \end{claim}

\begin{proof} Lemma~\ref{LemmaCommutator} implies that $\det \rho( t_{ab}^r) = 1$, whereas $\det \rho( t_{ab}^r) = \prod_{i=1}^n (1 + M_{ab}^r(i,i)) = 1 + \tr M_{ab}^r$ by Claim~\ref{claim-commConclusion}. \end{proof}

\begin{claim}\label{claim-[]} $\forall (a,b) \in \mf{J}, r\in R: \ \ M_{ab}^r(i,i)=0$ for $i \in \mf{I}\setminus\{a,b\}$. \end{claim}

\begin{proof} Consider $r \in R$ and $(a,b) \in \mf{J}$. If $c \in \mf{I} \setminus \{a,b\}$ then since $|\mf{I}|\geq 4$ we may choose $d \in \mf{I} \setminus \{a,b,c\}$. Let $U_c := \{ A \in \GL_n(S) \ | \ \forall x\in \mf{I}\setminus\{c\}: A(x,c), A(c,x) \in \mf{m}_S^l \}$. It is easy to see that $U_c$ is a group and $\chi: U_c \rightarrow S^\times$, $M \mapsto M(c,c)$ a group homomorphism (due to the fact that $(\mf{m}_S^l)^2=0$). Moreover, $\rho(t_{ab}^r)$, $\rho(t_{ad}^1)$,  $\rho(t_{db}^r) \in U_c$ and since $[t_{ad}^1, t_{db}^r] = t_{ab}^r$ by Lemma~\ref{RelLemma}, (\ref{rel[]}), we have that $\chi(\rho(t_{ab}^r)) = [\chi(\rho(t_{ad}^1))\ , \ \chi(\rho(t_{db}^r)) ] = 1$. We conclude that $M_{ab}^r(c,c)=0$. \end{proof}

\begin{claim}\label{claim-trConclusion} $\forall (a,b) \in \mf{J}, r \in R: \ \ M_{ab}^r(a,a) = -M_{ab}^r(b,b)$. \end{claim}

\begin{proof} This is an immediate consequence of Claim~\ref{claim-tr} and Claim~\ref{claim-[]}. \end{proof}

\begin{claim}\label{claim-finalRelations} If $a,b,c,d \in \mf{I}$ are such that $\{a,c\} \cap \{b,d\} = \emptyset$ and $(a,b) \neq (c,d)$ then for all $r,s \in R$: 
\[
\left\{
\begin{array}{cc}
p_s M_{ab}^r (a,c) &= - p_r M_{cd}^s (b,d), \\
p_s M_{ab}^r (d,b) & = - p_r M_{cd}^s(c,a).
\end{array} 
\right.
\]
\end{claim}

\begin{proof} Thanks to Claim~\ref{claim-commConclusion} and Claim~\ref{claim-[]} the formula of Claim~\ref{claim-comm} reduces to
 \[ p_s \Big( M_{ab}^r(a,c) e_{ad} \ - \ e_{cb} M_{ab}^r(d,b) \Big) \ = \ p_r \Big( M_{cd}^s(c,a) e_{cb} \ - \ e_{ad} M_{cd}^s(b,d) \Big), \]
If $(a,b) \neq (c,d)$ then the coefficients at $e_{ad}$ (resp. $e_{cb}$) on both sides must be equal. \end{proof}

\begin{claim} There exists $X \in M_n(\mf{m}_S^l)$ such that $\forall (a,b) \in \mf{J}, r\in R$ $\exists d_{ab}^r\in \mf{m}_S^l$: 
\[M_{ab}^r = p_r (e_{ab}X - X e_{ab}) + d_{ab}^r e_{ab}.\]\end{claim}

\begin{proof} Let $(a,b) \in \mf{J}$ and $c,d \in \mf{I}\setminus\{a,b\}$. The quadruple $(a,b,a,d)$ satisfies the assumptions of Claim~\ref{claim-finalRelations}, so $M_{ab}^1(a,a) = -M_{ad}^1(b,d)$ (the first relation). Combining with Claim~\ref{claim-trConclusion} we obtain $M_{ab}^1(b,b) = M_{ad}^1(b,d)$, so $M_{ay}^1(b,y)$ is independent of the choice of $y \in \mf{I}\setminus\{a\}$. We will denote this common value by $Y(b,a)$. Analogously, using the quadruple $(a,b,c,b)$ and the second relation of Claim~\ref{claim-finalRelations} we prove that the value of $M_{xb}^1(x,a)$, with $x$ ranging over $\mf{I}\setminus\{b\}$, is constant. We will denote it $X(b,a)$. 

Setting $X(a,a) := Y(a,a) := 0$ for all $a\in \mf{I}$ we obtain well defined matrices $X,Y \in M_n(\mf{m}_S^l)$. Since $M_{ab}^1(a,a)=-M_{ab}^1(b,b)$ by Claim~\ref{claim-trConclusion}, we have $X(b,a) = -Y(b,a)$ for all $(a,b)\in \mf{J}$, hence $X = - Y$. 

Consider $(a,b) \in \mf{J}$ and $c \in \mf{I}\setminus\{b\}$. Then it is possible to find $d \in \mf{I}$ such that $a,b,c,d$ satisfy the assumptions of Claim~\ref{claim-finalRelations} (choose any $d\in \mf{I}\setminus\{a,b\}$ if $a=c$ or $d=b$ if $a\neq c$). The first relation gives $\forall r\in R: M_{ab}^r(a,c) = -p_rM_{cd}^1(b,d) = p_r X(b,c)$. Similarily, $\forall d \in \mf{I}\setminus\{a\}, r\in R: M_{ab}^r (d,b) = -p_rX(d,a)$. We conclude that $M_{ab}^r = p_r (e_{ab}X - X e_{ab}) + M_{ab}^r(a,b) e_{ab}.$  \end{proof}

Let $X$ be as in the last claim and consider the representation $\tilde{\rho} := (I+X) \rho (I+X)^{-1}$. It follows that $\tilde{\rho} (t_{ab}^{p_r} + M_{ab}^r ) =
t_{ab}^{p_r} + M_{ab}^r + Xt_{ab}^{p_r} - t_{ab}^{p_r}X = t_{ab}^{ \phi_{ab}(r) }$, where $\phi_{ab}(r) := p_r + d_{ab}^r$. The lift $\tilde{\rho}$, strictly equivalent to $\rho$, satisfies thus (\ref{spec-form}) and Lemma~\ref{AuxLem} finishes the proof of the theorem.
\end{proof}

\begin{cor} Every $R \in \Ob(\cat)$ can be obtained as a universal deformation ring of a continuous representation of a profinite group satisfying $\Phi_p$. \end{cor}

\section{Lower dimensions}

In this section we continue working under Assumption~\ref{Ass} and discuss the possibility of extending Theorem~\ref{MainThm} to the cases $n=2$ and $n=3$. 

\subsection{Case $n=3$}

\begin{thm} \label{MainThmDim3}
Suppose $n=3$, $\kk \neq \GF_2$. Then $\iota$ is universal. 
\end{thm}

\begin{proof}
A closer look at the proof of Theorem~\ref{MainThm} shows that assuming Claim~\ref{claim-commConclusion} and Claim~\ref{claim-[]} the rest of the argument would hold also for $n\geq 3$. We provide thus a different argument for both of the claims in case $n=3$ and $\kk \neq \GF_2$ (this second assumption is actually needed only for proving Claim~~\ref{claim-[]}). In what follows, we assume $\mf{I} = \{a,b,c\}$.

A proof of Claim~\ref{claim-commConclusion}: Considering $(a,b) \in \mf{J}$, $r\in R$ we need to show that $M_{ab}^r(i,j)=0$ for $(i,j) \in \{ (b,a) , (c,a), (b,c) \}$. We see that the fact that $t_{ab}^r$ and $t_{ac}^1$ commute implies $M_{ab}^r(i,a)=0$ for $i \neq a$, just as in the case $n=4$. Similarily, the fact that $t_{ab}^r$ and $t_{cb}^1$ commute implies $M_{ab}^r(b,j)=0$ for $j\neq b$. 

A proof of Claim~\ref{claim-[]}: Let $(a,b) \in \mf{J}$, $r\in R$ and define $U_c$, $\chi$ just as in the case $n=4$. We need to show $M_{ab}^r(c,c)=0$. Making use of the assumption $\kk \neq \GF_2$ we choose $\lambda \in R$ such that $\lambda \not \equiv 0,1$ mod $\mf{m}_S$ and consider elements $d := d_{ac}^\lambda$, $t:=t_{ab}^{\frac{r}{\lambda-1}}\in U_c$. According to Lemma~\ref{RelLemma}, relation~(\ref{rel-conj}), $[d, t] = t_{ab}^{\frac{\lambda r}{\lambda-1}}t_{ab}^{\frac{-r}{\lambda-1}} = t_{ab}^r$. Evaluating $\chi$ at this element we conclude that $M_{ab}^r(c,c)=0$, just as in the case $n=4$. 
\end{proof}

As the following lemma shows, the case $\kk = \GF_2$ must really be excluded in Theorem~\ref{MainThmDim3}.

\begin{prop}\label{PropSpec3} Assume $n=3$ and $\kk = \GF_2$.
\begin{enumerate}[(i)]
\item There exists a lift $\rho_0$ of $\rhobar$ to $\mathbb{Z}_2$.
\item There is no $R \in \Ob(\cat)$ for which $\iota$ is universal. 
\end{enumerate}
\end{prop}

\begin{proof}

$(i)$ Since $\im \rhobar \subseteq \SL_3(\GF_2)$, it is enough to prove the claim for $R=\kk$. There exists an irreducible 3-dimensional representation of $\SL_3(\GF_2)$ over the ring $\mathbb{Z}[\omega]$, $\omega = \frac{-1+\sqrt{7}}{2}$,  defined in \cite{Atlas} by
\[ 
A := \begin{pmatrix} 1&1&0 \\ 0&1&0 \\ 0&0&1 \end{pmatrix} \ \mapsto \  \begin{pmatrix} 1 & \omega & -1-\omega \\ 0 & -1 & 0 \\ 0 & 0 &-1\end{pmatrix} 
\qquad, \qquad
B := \begin{pmatrix} 0&1&0 \\ 0&0&1 \\ 1&0&0 \end{pmatrix} \ \mapsto  \ \begin{pmatrix} 0&1&0 \\ 0&0&1 \\ 1&0&0 \end{pmatrix}. 
\]	
In order to check that a representation of $\SL_3(\GF_2) = \langle A, B \rangle$ may be defined this way recall that $\SL_3(\GF_2)$ is known to be isomorphic to $\PSL_2(\GF_7)$, which has an abstract presentation $\langle S, T \ | \  S^7 = T^2 = (ST)^3 = (S^4T)^4 =1 \rangle$ due to Sunday (\cite{Sunday}). One checks directly that the defining relations are satisfied both by $T := A$, $S := BA$ and their proposed images.  We obtain $\rho_0$ by sending $\omega$ to the root of $X^2+X+2$ that lies in $1+2\mathbb{Z}_2$.

$(ii)$ The claim follows immediately from part $(i)$ in case $\Char R \neq 0$. If $\Char R = 0$ consider the element $t := t_{12}^2 \in G$ and $f \in \Hom(R, \mathbb{Z}_2)$. Since $\rho_0(\rhobar(t)) = I$ and $\GL_n(f)(t) \neq I$,  we see that $\rho_0\circ \rhobar$ and $\iota_{*}(f) = \GL_n(f)\circ \iota$ are not strictly equivalent. Hence, $\iota$ is not universal.
\end{proof}

\subsection{Case $n=2$} \ 

We use a totally different approach. Given $R \in \Ob(\cat)$ the group $\SL_n(R)$ contains a subgroup $D_R := \{  \M{a}{}{}{a^{-1}} \ | \ a \in \mu_R \}$. Since $D_R$ is the isomorphic image of $D_{W(\kk)}$ under the unique $\cat$-morphism $W(\kk)\rightarrow R$, we will identify the groups $D_R$ obtained for different $R\in \Ob(\cat)$ and write $D$ for each of them.

\begin{lem} \label{SpecFormLem}
If $n=2$ and $S \in \Ob(\cat)$ then every deformation of $\wbar{\rho}$ to $S$ has a representative $\rho$ such that $\rho|_D = id_D $.
\end{lem}

\begin{proof}
The group $D$ has order $|\kk|-1$, which is coprime to $p$, so by Lemma~\ref{ProjLemma} representation $\rhobar|_{D}$ has precisely one deformation to $S$.	On the other hand, $id_D$ clearly is a lift of $\rhobar|_{D}$. We conclude that given a lift of $\rhobar$ to $S$, we may conjugate it by a matrix in $I+M_n(\mf{m}_S)$ in order to obtain a strictly equivalent lift $\rho$ with the property $\rho|_D = id_D$. 
\end{proof}

\begin{thm} \label{MainThmDim2}
Suppose $n=2$ and $\kk \neq \GF_2$, $\GF_3$, $\GF_5$. Then $\iota$ is universal.
\end{thm}

\begin{proof}
Let $S \in \Ob(\cat)$. According to Lemma~\ref{AuxLem} we only need to show that every lift $\rho$ of $\wbar{\rho}$ to $S$ is strictly equivalent to $\rho_f := \GL_n(f)|_G$ for some $f \in \Hom_{\cat} (R,S)$. Applying Lemma~\ref{SpecFormLem} we will restrict to the $\rho$ satisfying $\rho|_D = id_D$. Using the assumption $\kk \neq \GF_2$, $\GF_3$, $\GF_5$ let us choose $\alpha \in \mu_R$ such that $\alpha^4 \neq 1$, set $\delta:=d(\alpha,\alpha^{-1}) \in D$ and note that $\rho( \delta ) = \delta$. Finally, we introduce the convention of writing $t^r$ in place of $t_{12}^r$. 

Consider $t^r \in G$, $r \in R^\times$ and suppose $\rho(t^r) = \M{a}{b}{c}{d}$. Then due to Lemma~\ref{RelLemma}, (\ref{rel-conj}):
\[ \rho(t^{\alpha^2r}) = \rho(\delta t^r \delta^{-1}) = \delta \MM{a}{b}{c}{d} \delta^{-1} = \MM{a}{ b\alpha^2}{c\alpha^{-2}}{d}. \]
As $t^r$ and $t^{\alpha^2r}$ commute, so do their lifts. This implies $a^2 + \alpha^2 bc = a^2 + \alpha^{-2} bc$, i.e., $(\alpha^4-1)bc = 0$ and consequently $c=0$, since $(\alpha^4-1)b$ is invertible. Thus $\rho(t^r)$ is an upper triangular matrix for every $r \in R^\times$. By Lemma~\ref{RelLemma}, (\ref{rel-conj}) we have $t^r = [ \delta, t^s]$ for $s := r(\alpha^2-1)^{-1}$, so $\rho(t^r)= [\delta, \rho(t^s)]$. Both $\delta$ and $\rho(t^s)$ are upper triangular, so we conclude that the diagonal entries of $\rho(t^r)$ are equal to $1$, i.e., there exists $f : R^\times \rightarrow S$ such that $\rho(t^r) = t^{f(r)}$ for all $r\in R^\times$. Since $R^\times + R^\times = R$, this result immediately extends to the whole of $R$. We similarily obtain an analogous result for $\rho(t_{21}^r)$ and see that lift $\rho$ satisfies condition (\ref{spec-form}) of Lemma~\ref{AuxLem}, hence $[\rho] \in \iota_{*}(h_R(S))$.
\end{proof}

\begin{prop}\label{PropSpec2} Assume $n=2$ and $\kk \in \{ \GF_2, \GF_3, \GF_5 \}$. 
\begin{enumerate}[(i)]
\item There exists a lift $\rho_0$ of $\rhobar$ to, respectively, $\mathbb{Z}_2$, $\mathbb{Z}_3$ or $\mathbb{Z}_5[\sqrt{5}]$. 
\item There is no $R \in \Ob(\cat)$ for which $\iota$ is universal.
\end{enumerate}
\end{prop}

\begin{proof}

$(i)$ It is enough to prove the claim for $R=\kk$. One easily checks that $\SL_2(\GF_2) = \langle \tau \rangle \rtimes \langle  \eps \rangle$, where $\tau := \M{0}{1}{1}{1}$, $\eps := \M{0}{1}{1}{0}$, and that $\tau \mapsto \M{\phantom{-}0}{\phantom{-}1}{-1}{-1}$, $\eps \mapsto \M{0}{1}{1}{0}$ defines a lift of $\SL_2(\GF_2)$ to $\mathbb{Z}_2$. In case $p\in\{3,5\}$  it is known (\cite[\S 7.6]{Coxeter}) that $\SL_2(\GF_p)$ has presentation \[\langle A,B,C \ | \ A^p=B^3=C^2=ABC \rangle = \langle A, C \ | \ A^p = (A^{-1}C)^3 = C^2 \ \rangle,\] realized for example by the following choice of generators: $A := \M{-1}{\phantom{-}0}{-1}{-1}, C:=\M{\phantom{-}0}{1}{-1}{0}.$
Using this fact and defining $t\in \mathbb{Z}_3$  by $t^2=-2$, $t \equiv 2 (\hbox{mod } 3)$ it is easy to check that
\[ \MM{-1}{\phantom{-}0}{-1}{-1} \mapsto \frac{1}{2}\MM{1}{t+1}{t-1}{1} \quad , \quad \MM{\phantom{-}0}{1}{-1}{0} \mapsto \MM{\phantom{-}0}{1}{-1}{0} \]
extends to a lift of $\SL_2(\GF_3)$ to $\mathbb{Z}_3$. Similarily, defining $i,\varphi \in \mathbb{Z}_5[\sqrt{5}]$ by $i^2=-1$, $i \equiv 2 (\hbox{mod }5)$ and $\varphi := \frac{1+\sqrt{5}}{2}$ we have that
\[ \MM{-1}{\phantom{-}0}{-1}{-1} \mapsto \frac{1}{2}\MM{\varphi}{i(\varphi-1)+1}{i(\varphi-1)-1}{\varphi} \quad , \quad \MM{\phantom{-}0}{1}{-1}{0} \mapsto \MM{\phantom{-}0}{1}{-1}{0} \]
extends to a lift of $\SL_2(\GF_5)$ to $\GL_2(\mathbb{Z}_5[\sqrt{5}])$.

$(ii)$ The claim follows from $(i)$ by an argument analogous to the one used in Proposition~\ref{PropSpec3}.
\end{proof}

\begin{cor}
Combining Theorems~\ref{MainThm}, \ref{MainThmDim3}, \ref{MainThmDim2} and Propositions~\ref{PropSpec3}, \ref{PropSpec2} we obtain Theorem~\ref{IntrThm} stated in the introduction to this paper.
\end{cor}

\section{Special cases}

It would be interesting to know what are the universal deformation rings of~$\rhobar$ in the cases not treated by Theorem~\ref{MainThmDim3} and Theorem~\ref{MainThmDim2}. We provide a complete answer in case $R=\kk$. 

\begin{prop}\label{PropDefo3}
The universal deformation ring of $\rhobar$ for $n=3$ and $R = \GF_2$ is $\mathbb{Z}_2$.
\end{prop}

\begin{proof}
It is sufficient to show that the versal deformation ring $R_v$ of $\rhobar$ is a quotient of $\mathbb{Z}_2$. Indeed, by Lemma~\ref{PropSpec3} it can not be a proper quotient of $\mathbb{Z}_2$, so it will follow that $R_v =\mathbb{Z}_2$
(which also implies that $R_v$ is universal).

By Proposition~\ref{PropPresentation} we need to check that the tangent space to $\Def_{\wbar{\rho}}$ is zero dimensional which is equivalent to checking that every deformation of $\rhobar$ to $S=\GF_2[\eps]$ is induced by a $\cat$-morphism $R\rightarrow S$. This can be done by a modification of the argument used in the inductive step of the proof of Theorem~\ref{MainThm}, in the special case $S=\GF_2[\eps]$, $R=\GF_2$. As in the proof of Theorem~\ref{MainThmDim3}, we only have to provide an argument for Claims~\ref{claim-commConclusion} and \ref{claim-[]}.

The proof of Claim~\ref{claim-commConclusion} given in Theorem~\ref{MainThmDim3} holds true also in this case. As for Claim~\ref{claim-[]}, suppose $\mf{I} = \{a,b,c\}$ and consider $r \in R$. We need to check that $M_{ab}^r(c,c)=0$. Since $R =\GF_2$, we only have to consider the case $r=1$, in which we will write $t_{ab}$ for $t_{ab}^1$ and $M$ for $M_{ab}^1$. Note that $t_{ab}$ is of order $2$ and so does its lift $\rho(t_{ab}) = t_{ab} + M$. Using the fact that $\Char S = 2$ we compute: \[ I = (t_{ab}+M)^2 = I + t_{ab}M + Mt_{ab} = I + e_{ab}M + Me_{ab}\]
In particular, comparison of $(a,b)$-entries yields: $M(a,a) + M(b,b)=0$. Since $M(a,a) + M(b,b) + M(c,c) = \tr M = 0$ by Claim~\ref{claim-tr}, we conclude that $M_{ab}^1(c,c)=0$.
\end{proof}

\begin{lem} \label{LemmaChebyshev}
Let polynomials $f_n \in \mathbb{Z}[X]$ be defined recursively by $f_0 = 0$, $f_1 = 1$, $f_{n+1} = Xf_n - f_{n-1}$. Consider a commutative ring $R$ and a matrix $M \in M_2(R)$ such that $\det M = 1$ and at least one of its off-diagonal entries is not a zero divisor. If $n=2k+1$ is an odd positive integer then $M^n=-I$ holds if and only if $t := \tr M$ is a root of the polynomial $f_{k+1}-f_k$. 
\end{lem}
\begin{proof}
By Cayley-Hamilton we have $X^2 = tX - 1$ and it is easy to check that $\forall n\geq 1: \ X^n = f_n(t) X - f_{n-1}(t) I$. It follows that $X^n = -I$ if and only if $f_n(t) = 0$ and $f_{n-1}(t) = 1$. If $I_n =(f_n, f_{n-1}-1)$ is the ideal of $\mathbb{Z}[X]$ generated by $f_n$ and $f_{n-1}-1$, then one easily proves by induction on $l$ that $\forall l \in \{0,\ldots, n-1\}: \ I_n = ( f_{n-l}-f_{l} , f_{n-1-l}- f_{l+1})$. In particular, for $l=k$ we obtain $I_n = (f_{k+1}-f_{k})$. 
\end{proof}

\begin{prop} \label{PropDefo2}
Assume $n=2$ and $\kk \in \{ \GF_2, \GF_3, \GF_5 \}$. The universal deformation rings of $\rhobar$ for $R = \kk$ are, respectively: $\mathbb{Z}_2$, $\mathbb{Z}_3[X]/(X^3-1)$ and $\mathbb{Z}_5[\sqrt{5}]$. 
\end{prop}

\begin{proof}
For $G = \SL_2(\GF_2)$ we observe that the $\GF_2 G$-module $V_{\rhobar}$ is projective. Indeed, for a field $\kk$ of characteristic $p$ and a finite group $G$ with $p$-Sylow subgroup $S$, a $\kk G$-module $V$ is projective if and only if $V$ is projective as a $\kk S$-module (\cite[p. 66, Corollary 3]{Alperin}). In this case $S \cong \langle\M{0}{1}{1}{0}\rangle\cong C_2$ (cyclic group of order 2) and $V_{\rhobar |_S} \cong \GF_2[C_2]$ is even a free $\GF_2 S$-module. The claim follows now from Lemma~\ref{ProjLemma}. 

\medskip

The case $G = \SL_2(\GF_3)$ can be approached via Lemma~\ref{Lemma_sl-gl}. It follows from the discussion in \cite[\S 7.2, \S 7.6]{Coxeter} that $G \cong G' \rtimes C_3$, where $G' = \left\langle \M{1}{1}{1}{2}, \M{0}{1}{2}{0} \right\rangle$ is isomorphic to the quaternion group of order 8. Since $G'$ has order coprime to $p=3$, it follows from Lemma~\ref{ProjLemma} that $\mathbb{Z}_3$ is the universal deformation ring for $\rhobar|_{G'}$. Proposition~\ref{PropSpec2} shows that there exists a universal lift of $\rhobar|_{G'}$ that may be extended to $G$. Moreover, it is easy to check that $\Ad(\rhobar)^{G'} = \kk I$. Thus, given that $G/G' \cong C_3$, the universal deformation ring of $\rhobar$ is $\mathbb{Z}_3[C_3] \cong \mathbb{Z}_3[X]/(X^3-1)$. 

\medskip

In case $R = \GF_5$ we will simply check that the lift described in Lemma~\ref{PropSpec2} is universal.  Consider $S \in \Ob(\cat)$, $\xi \in \Def_{\rhobar}(S)$ and let $A := \M{-1}{\phantom{-}0}{-1}{-1}$, $C:=\M{\phantom{-}0}{1}{-1}{0} \in \SL_2(R)$; we moreover identify $C$ with $\M{\phantom{-}0}{1}{-1}{0} \in \SL_2(S)$. Since $H:=\langle C \rangle$ is of order $4$, it follows from Lemma~\ref{ProjLemma} that there is precisely one deformation of $\rhobar|_H$  to $S$. Hence, $\xi$ has a representative $\rho\in \xi$ satisfying $\rho(C) = C$. We claim that there is precisely one $\rho \in \xi$ satisfying this condition and such that the diagonal entries of $\rho(A)$ are equal.  Indeed, if $\rho(A) = \M{a}{b}{c}{d}$ and $X \in I+M_n(\mf{m}_S)$ is a matrix commuting with $\M{0}{1}{-1}{0}$ then there exist $u,v \in S$ such that $X = \M{u}{v}{-v}{u}$ and writing $t := v/u \in \maxId{S}$ we obtain
\[ \MM{u}{v}{-u}{v} \MM{a}{b}{c}{d} \MM{u}{v}{-v}{u}^{-1} = \frac{1}{1+t^2}\MM{a+ct+bt+dt^2}{b+dt-at-ct^2}{c-at+dt-bt^2}{d-bt-ct+at^2}.\]
The equation $a+ct+bt+dt^2 = d-bt-ct+at^2$ is equivalent to $t^2 (d-a) + 2t(b+c) + (a-d) = 0$ and has precisely one solution $t \in \mf{m}_S$, due to Hensel's lemma.

Since $A$ and $C$ generate $G$, a lift $\rho$ is uniquely determined by $\rho(A)$ and $\rho(C)$. Note that $\det \rho(A) = \det \rho(C) =1$ due to Lemma~\ref{LemmaCommutator}. Using all the above observations and the presentation of $\SL_2(\GF_5)$ introduced in the proof of Lemma~\ref{PropSpec2}, we see that deformations of $\rhobar$ to $S$ correspond bijectively with matrices $M = \M{a}{b}{c}{a} \in M_n(S)$ such that $\M{a}{b}{c}{a} \equiv \M{-1}{0}{-1}{-1} \hbox{ mod } \maxId{S}$, $\det M=1$ and $M^5 = \left(M^{-1}C\right)^3 =-I$. 

 By Lemma~\ref{LemmaChebyshev}, the last condition is equivalent to $\tr M = 2a$ being a root of $f_3-f_2=X^2-X-1$ and $\tr (M^{-1}C)=b-c$ being a root of $f_2-f_1 = X-1$. If $(2a)^2=2a+1$ then solving the quadratic equation $b(b-1)+1-a^2 = 1-\det M =0$ we obtain that $b = \frac{1-i(1-2a)}{2}$ with $i^2=-1$. We conclude that the full set of conditions imposed on $a,b,c$ is as follows: $a =\frac{\varphi}{2}$, $b = \frac{1-i(1-\varphi)}{2}$, $c=b-1$, where $\varphi^2=\varphi+1$, $\varphi \equiv -2  \hbox{ mod } \maxId{S}$ and $i^2=-1$, $i \equiv 2  \ \hbox{mod } \maxId{S}$. It follows that every deformation of $\rhobar$ to $S$ is induced by a morphism $\mathbb{Z}_5[\sqrt{5}] \rightarrow S$ applied to the universal lift defined in the proof of Lemma~\ref{PropSpec2}.
\end{proof}

\begin{rmk}
The above results obtained for $n=2$ seem to be not entirely new. For example, Rainone in \cite{Rainone} has considered the case $\kk=\GF_2$ and Mazur mentions the case $\kk=\GF_5$ in \cite[ \S 1.9]{Mazur2} though without giving a proof. Also Bleher and Chinburg obtained analogous results for an algebraically closed field in \cite{BleherChinburg3}. However, there does not seem to be an easy and complete treatment of all the cases in the literature.
\end{rmk}

\begin{rmk} It is worth noting that even though in case $\kk=\GF_2$ we have obtained the same universal deformation ring for $n=2$ and $n=3$, the $\kk G$-module $V_{\rhobar}$ is not projective when $n=3$. Indeed, it is known (\cite[p. 33, Corollary 7]{Alperin}) that if a $\kk G$-module $V$ is projective then the order of the $p$-Sylow subgroup $S$ of $G$ divides $\dim_\kk V$. Here $|S|=8$ and $\dim_\kk V_{\rhobar}=3$. 
\end{rmk}

\section{The general linear group}

It is a very natural question to ask what results would be obtained in the preceding sections if we considered the general linear group instead of the special linear one. Let us note for example that Rainone studied in \cite{Rainone} the deformations of the identity map $\GL_2(\GF_p) \rightarrow \GL_2(\GF_p)$ and obtained $\GF_p$ as the universal deformation ring for all $p>3$. 

Let $R \in \Ob(\cat)$ and $\rhobar : \GL_n(R) \rightarrow \GL_n(\kk)$ be induced by the reduction $R \twoheadrightarrow \kk$. Let us consider generally 
\[ \mf{F} \ := \ \{ G \leq \GL_n(R) \ | \ G \hbox{ closed and } \SL_n(R) \leq G \} \]
Note that the determinant map gives a bijective correspondence between $\mf{F}$ and closed subgroups of $R^\times$. In particular, every $G \in \mf{F}$ is a normal subgroup of $\GL_n(R)$. Moreover, since $\GL_n(R) = \varprojlim \GL_n(R/\maxId{R}^k)$ is profinite, so are all elements of $\mf{F}$. 

If the finiteness condition $\Phi_p$ is safisfied for $G \in \mf{F}$ then the universal deformation ring of $\rhobar|_G$ exists by Lemma~\ref{End1Lem} and the criterion of section \ref{IntroExistCrit}. However, in general $\Phi_p$ need not hold. 

\begin{exm}
Let $n\geq 3$, $R = \GF_p[[X]]$ and $G=\GL_n(R)$. One may check that $R_1^\times \cong \mathbb{Z}_p^\mathbb{N}$. Using the determinant map and isomorphism $R^\times \cong \mu_R \oplus R_1^\times$ we obtain that $\CHom(G, \mathbb{Z}/p\mathbb{Z})$ is infinite. 
\end{exm}

Consequently, $\Def_{\rhobar|_G}$ may be not representable over $\cat$. If it is, we will denote an object representing it by $R_u(G)$. 

\begin{prop}\label{PropH-G}
Let $G,H \in \mf{F}$ be such that $H \subseteq G$ and $\Def_{\rhobar|_H}$ is representable. If the pro-$p$ completion $(G/H)^p$ is topologically finitely generated then $\Def_{\rhobar|_G}$ is represented by $R_u(H)[[ (G/H)^p ]]$. Otherwise it is not representable over $\cat$. 
\end{prop}
\begin{proof}
It is an immediate consequence of Lemma~\ref{Lemma_sl-gl} since for an abelian pro-$p$ group being topologically finitely generated is equivalent to being finitely generated as a $\mathbb{Z}_p$-module.
\end{proof}

We conclude that the results regarding $\GL_n(\kk)$ generalize much better and in a more natural way to the group $\mu L_n(R)$ defined below than to $\GL_n(R)$:

\begin{cor}\label{CorollaryGL} Suppose $(n, \kk) \not \in \{  (2, \GF_2) , (2, \GF_3) , (2, \GF_5),  (3,\GF_2)  \}$.
\begin{enumerate}[(i)]
\item Either $R_u(\GL_n(R)) \cong R[[  R^\times_1  ]]$ or $\Def_{\rhobar}$ is not representable over $\cat$. In particular, $R$ represents $\Def_{\rhobar}$ if and only if $R = \kk$. 
\item Let $\mu L_n(R) := \{ A \in \GL_n(R) \ | \ \det A \in \mu_R \}$. The set $\mf{G}$ of all $G \in \mf{F}$ for which $\Def_{\rhobar|_G}$ is represented by $R$ coincides with the set $\{ G \in \mf{F} \ | \ G \leq \mu L_n(R)\}$.
\end{enumerate}
\end{cor}
\begin{proof}
$(i)$ By Theorems~\ref{MainThm}, \ref{MainThmDim3} and \ref{MainThmDim2} we have  $R_u(\SL_n(R)) = R$, so we may apply Proposition~\ref{PropH-G} with $G = \GL_n(R)$ and $H= \SL_n(R)$. Since $G / H \cong R^\times \cong \mu_R \ \oplus \ R^\times_1$, $\mu_R$ is of finite order coprime to $p$ and $ R^\times_1$ is a pro-$p$ group, we have that $(R^\times)^{p} \cong R^\times_1$. Hence, the first claim follows.

$(ii)$ A similar reasoning as in part $(i)$ shows that elements of $\mf{G}$ correspond (via the determinant map) with these closed subgroups of $R^\times \cong \mu_R \ \oplus \ R^\times_1$ that are topologically finitely generated and have a trivial pro-$p$ completion. Since $R^\times_1$ is a pro-p group and $\mu_R$ finite of order coprime to $p$, every closed subgroup of $\mu_R \ \oplus \ R^\times_1$ is a product $A\oplus B$ of closed subgroups $A\leq \mu_R$ and $B \leq R^\times_1$. Moreover, $(A\oplus B)^p \cong B$, so the elements of $\mf{G}$ correspond with subgroups of $\mu_R$. 
\end{proof}

\begin{rmk}
For $(n, \kk) \in \{  (2, \GF_2) , (2, \GF_3) , (3,\GF_2)  \}$ there exists a lift of $\GL_n(\kk)$ to $\mathbb{Z}_p$. Indeed, in case $\kk = \GF_2$ we have $\GL_n(\kk) = \SL_n(\kk)$, so we already know it; for $n=2$, $\kk = \GF_3$ see \cite{Rainone} (it is also not difficult to check it directly, knowing that $\SL_2(\GF_3)$ lifts to $\mathbb{Z}_3$). This fact and a reasoning as in Proposition~\ref{PropSpec3} show that $R$ does not represent $\Def_{\rhobar|_G}$ for any $G \in \mf{F}$. 

If $n=2$, $\kk = \GF_5$ then $R_u( \SL_n(R) ) \not \cong R$, but it is true that $R_u( \mu L_n(R) ) \cong R$. This can be shown slightly modifying the proof of Theorem~\ref{MainThmDim2}. Omitting the details, the idea is to consider the subgroup $\{ \M{a}{}{}{1} | \ a \in \mu_R \} \leq \mu L_n(R)$, instead of $\{ \M{a}{}{}{a^{-1}} | \ a \in \mu_R\} \leq \SL_n(R)$ and to choose $\alpha \in \kk^\times $ satisfying the condition $\alpha^2 \neq 1$, instead of $\alpha^4 \neq 1$. Assuming this result, Proposition~\ref{PropH-G} implies that Corollary~\ref{CorollaryGL} holds in the case $n=2$, $\kk = \GF_5$ as well.
\end{rmk}

\end{document}